\documentclass[12pt]{amsart}
\usepackage{amsmath}
\usepackage{amssymb}
\usepackage{amsfonts}
\usepackage{textcomp}
\usepackage{amsthm}
\usepackage{mathrsfs}
\usepackage[latin1]{inputenc}
\usepackage[all]{xy}
\usepackage{hyperref}
\usepackage{url}
\usepackage{graphicx}
\usepackage{tikz-cd}
\usepackage[paperheight=11in,paperwidth=8.5in,left=1in,top=1.25in,right=1in,bottom=1.25in,]{geometry}

\newcommand{\HH}{\mathbb{H}^3}
\newcommand{\RR}{\mathbb{R}}
\newcommand{\MM}{\mathcal{M}}
\newcommand{\II}{{\rm I}\!{\rm I}}
\newcommand{\VR}{{\rm V}_{\rm R}}
\newcommand{\VC}{{\rm V}_{\rm C}}

\newtheorem{theorem}{Theorem}[section]
\newtheorem{lem}{Lemma}[section]

\newtheorem{defi}{Definition}[section]

\newtheorem{que}{Question}[section]
\newtheorem{remark}{Remark}[section]

\title{Upper bounds on Renormalized Volume for Schottky groups}
\author{Franco Vargas Pallete}
\thanks{Research supported by the Minerva Research Foundation}
\address{Institut des Hautes \'Etudes Scientifiques, Bures-sur-Yvette, France}
\email{\protect\url{vargaspallete@ihes.fr}}
\begin{document}

\maketitle

\begin{abstract}
    In this article we show that for any given Riemann surface $\Sigma$ of genus $g$, we can find an upper bound for the renormalized volume of a (hyperbolic) Schottky group with boundary at infinity conformal to $\Sigma$. This bound depends on the genus of $\Sigma$ and the combined extremal lengths on $\Sigma$ of $(g-1)$ disjoint, non-homotopic, simple closed compressible curves, whose complement is the union of genus $1$ components. This bound on $\VR$ is used to partially answer a question posed by Maldacena about comparing renormalized volumes of Schottky and Fuchsian manifolds with the same conformal boundary.
\end{abstract}

\section{Introduction}\label{sec:intro}

Renormalized volume (denoted by $\VR$) is a geometric quantity motivated by the AdS/CFT correspondence and the calculation of gravitational action (\cite{Witten98}). We will describe the setup for our main question intuitively, delaying more specific details until Section \ref{sec:background}. In the case of convex co-compact hyperbolic manifolds, we can define $\VR$ as follows: we start with the divergent integral of the volume form, take an exhaustion by compact sets of the manifold and apply a process of renormalization, so we can rescue a finite quantity that is geometrically meaningful. This finite quantity is what we denote by $\VR$. People interested in gravitational action ask the following question \cite{Malda}, which is also a natural geometric question to consider: Given a fixed conformal manifold $\Sigma$, and two model geometries $M_1, M_2$ (not necessarily hyperbolic) that are asymptotic to $\Sigma$ at infinity, which one of the quantities $\VR(M_1),\,\VR(M_2)$ is the largest/smallest? In this article we give a partial answer to this type of question (posed by Maldacena via personal communication to the author) in the context of hyperbolic $3$-manifolds, where we particularly compare some of the simplest cases to consider, the Schottky and Fuchsian models. This is accomplished by showing explicit upper bounds on $\VR$ for Schottky manifolds and then comparing such bounds to the renormalized volume of Fuchsian manifolds.                                                                                     

The article is organized as follows. In Section \ref{sec:background} we give the appropriate background for renormalized volume $\VR$ and motivate Maldacena's question about bounds for $\VR$. It also includes the formulation for our partial answer. Section \ref{sec:prel} describes preliminary results we will need in order to bound $\VR$. In particular, we discuss the bounds of $\VR$ in terms of the volume of the convex core $\VC$ and the bending lamination. We setup as well the isoperimetric inequalities that hold in hyperbolic $3$-space. In Section \ref{sec:bounds} we will describe how these isoperimetric inequalities of Section \ref{sec:prel} give a bound for $\VC$ in terms of extremal lengths of the conformal boundary. This bound will be optimized by compressing the shortest set of $g-1$ curves, namely the ones with minimal sum of square roots of extremal lengths, as long as the compression yields a union of tori. We end by proving our main result, which is a bound depending only on genus and extremal length of compressing curves, while also giving conditions on the shortest curves to show that $\VR$ is smaller in the Schottky model.

\textbf{Acknowledgements} I would like to thank Juan Maldacena for bringing this problem to my attention, as well as for our much appreciated discussions about the topic. I would also like to thank Yair Minsky for his very helpful advice and suggestions. I am also very thankful to Didac Martinez-Granado, Robin Neumayer and Celso Viana for our conversations while working on this problem, as well for the helpful feedback from the anonymous referee.

\section{Background}\label{sec:background}
 
Let us start by defining convex co-compact hyperbolic manifolds.

\begin{defi}
We say that a hyperbolic $3$-manifold $M=\mathbb{H}^3/\Gamma$ is convex co-compact if there exists convex compact submanifold $N\subset M$ so that the inclusion map $N\hookrightarrow M$ is a homotopic equivalence.
\end{defi}

The next two well-known examples will be our main focus of study.

\noindent\textbf{Examples:} 
\begin{enumerate}
    \item\label{example:fuchsian} For Riemann surface $\Sigma=(S,h)$, consider the manifold $M=S\times\mathbb{R}$ with hyperbolic metric $ \cosh^2(t)h + dt^2$. Then for any $a<0<b$ it is an easy exercise to verify that the submanifold $N_{a,b}=S\times[a,b]$ is convex. Such manifolds $M$ are known as \textit{Fuchsian}.
    
    \item\label{example:schottky} Let $\{D^\pm_j\}$ a finite disjoint collection of closed (topological) disks in $\overline{\mathbb{C}}$ and let $\{ g_j\}$ a collection of M\"obius maps so that $g_j(\overline{\mathbb{C}\setminus D^+_j}) = D^-_j$. Then the group $\Gamma$ generated by $\lbrace g_j \rbrace$ is freely generated and $M=\mathbb{H}^3/\Gamma$ has the topology of a handlebody. Such manifolds $M$ are known as \textit{Schottky}.
\end{enumerate}

Renormalized volume for convex co-compact hyperbolic $3$-manifolds (as described in \cite{KS08}) is motivated by the computation of the gravity action $S_{gr}[g]$ in the context of the Anti-de-Sitter/Conformal Field Theory (AdS/CFT) correspondence (\cite{Witten98}). For an Einstein manifold $(M,g)$ one would like to calculate the following expression:

\begin{equation}
    S_{gr}[g] := -\int_M(R-2\Lambda)dv - 2\int_{\partial M}\II,
\end{equation}
where $R$ is the scalar curvature of $(M,g)$, $\II$ is the second fundamental form of $\partial M$, $\Lambda$ is the cosmological constant, and $dv$ correspond the volume form in $M$.

For hyperbolic $3$-manifolds we have that $R=-6$. Moreover, $\Lambda$ relates to the radius of curvature by $l=1/\sqrt{-\Lambda}$, so $\Lambda =-1$. Hence the gravity action has the simpler expression

\begin{equation}\label{eq:gactionhyp}
    S_{gr}[g] = 4\int_M dv + 2\int_{\partial M}Hda = 4\left(\int_M dv - \frac12 \int_{\partial M}\II\right).
\end{equation}
Although this integral diverges for $M$ convex co-compact, we can still make sense of the calculation in (\ref{eq:gactionhyp}) by a process of \textit{renormalization}. This allows us to understand how the integral blows up as we exhaust $M$ by compact subsets, and then rescue a number out of it. Moreover, we want to do it in such a way that $S_{gr}[g]$ is a function on the conformal boundary. Given our geometric approach, we will describe how to do so for $ {\rm vol}(M) - \frac12 \int_{\partial M}Hda$, which ends up having the same renormalization as $\frac14 S_{gr}[g]$.

We can define the $W$-$volume$ of a compact, convex $C^{1,1}$-submanifold $N\subset M$ as

\begin{equation}
    W(M,N)= {\rm vol}(N) - \frac12 \int_{\partial N}Hda,
\end{equation}
where $H$ stands for mean curvature (i.e. the arithmetic mean of the principal curvatures at a given point).

Given a convex co-compact hyperbolic $3$-manifold $M$, we define its \textit{domain of discontinuity} $\Omega(M)$ as the largest set of $\partial_\infty\mathbb{H}^3$ where $\pi_1(M)$ acts properly discontinuously. We define the \textit{conformal boundary at infinity} of $M$, denoted by $\partial_\infty M$, as the quotient by $\pi_1(M)$ of the domain of discontinuity of $\pi_1(M)$ on $\partial_\infty\mathbb{H}^3$. Since $\pi_1(M)$ acts properly discontinuously by conformal maps on its domain of discontinuity $\Omega(M)$, $\partial_\infty M$ has a naturally defined conformal structure. We  refer to this conformal structure as the \textit{conformal class at infinity} of $M$. By an analogous reasoning, $\partial_\infty M$ has a naturally defined projective structure, meaning local charts to $S^2$ with transition maps given by M\"obius transformations.

\noindent\textbf{Examples:} 
\begin{enumerate}
    \item\label{example:fuchsianinfinity} If $M$ is a Fuchsian manifold obtained by a Riemann surface $S$ as in Example \ref{example:fuchsian}, then $\partial_\infty M$ is the disjoint union of $S$ and $S'$, where $S'$ is the Riemann surface obtained by reversing the orientation of $S$.
    
    \item\label{example:schottkyinfinity} If $M$ is a Schottky manifold obtained by disjoint closed disks $\lbrace D^\pm_j \rbrace$ and M\"obius transformation $\lbrace g_j\rbrace$ as in Example \ref{example:schottky}, then $\partial_\infty M$ is given by taking the quotient of the closure of $\overline{\mathbb{C}}\setminus\left( \cup_j D^\pm_j \right)$ by the boundary identifying maps $g_j:\partial D^+_j \rightarrow \partial D^-_j$.
\end{enumerate}

Given a Riemannian metric $h$ in the conformal class of $\partial_\infty M$, Epstein (\cite{Epstein}) constructs a family convex submanifolds $N_r$ (for $r$ sufficiently large depending on the metric) with equidistant boundaries by taking envelopes of families of horospheres. Such a family of equidistant submanifolds depends on the projective structure of $(\partial M, h)$ and exhausts the manifold $M$. Because the boundaries are equidistant, $W$-volume along the family $N_r$ has the property ([\cite{Schlenker13}, Lemma 3.6])

\begin{equation}\label{eq:equidist}
    W(M,N_r) = W(M,N_s) - \pi(r-s)\chi(\partial M)
\end{equation}
which leads to the following definition

\begin{equation}
    W(M,h) := W(M,N_r) + \pi r\chi(\partial M).
\end{equation}
Observe that $W(M,h)$ is well-defined since by (\ref{eq:equidist}) the right-hand side does not depend on $r$.

By taking $h_{\rm hyp}$, the unique metric of constant curvature in the given conformal class at infinity, we define Renormalized Volume $\VR$ as

\begin{equation}
    \VR(M) = W(M,h_{\rm hyp})
\end{equation}

The relationship between the conformal metric at infinity and the equidistant foliation is defined in such a way that if we denote the metric in $\partial N_r$ by $h_r$ and identify the boundaries $\partial N_r$ by the normal geodesic flow, then the metric $h$ at infinity is given by $h=\lim_{r\rightarrow\infty} 4e^{-2r}h_r$. This choice of constant is done so the following example works as described.

\textbf{Example:} Given a Riemann surface $\Sigma=(S,h)$, consider the Fuchsian manifold $M=S\times\mathbb{R}$ with hyperbolic metric $ \cosh^2(t)h + dt^2$. The equidistant foliation for the hyperbolic metric at infinity (which is isometric to $\Sigma$) coincides with the product foliation of $S\times\mathbb{R}$. Moreover (and here is where the relation between conformal metric and foliation is determined) $N_0$ collapses to be equal to the totally geodesic surface $S\times\lbrace0\rbrace$. Since then both ${\rm vol}(N_0)$ and $H_{\partial N_0}$ vanish, we have:

\begin{equation}\VR(S\times\mathbb{R}, \cosh^2(t)h + dt^2) = W(M,N_0) = 0\\
\end{equation}

Maldacena's question \cite{Malda} asks, for a fixed Riemann surface $\Sigma$ (with the topological type of a closed surface $S_g$ of genus $g$), to compare the gravity actions of (potentially disconnected) hyperbolic $3$-dimensional fillings of 2 copies of $\Sigma$ with opposite orientation. More precisely, one could take $M$ to be

\begin{enumerate}
    \item\label{qfcase} A hyperbolic metric in $S\times\mathbb{R}$, such that the conformal boundary $\partial_\infty M$ are the Riemann surface $\Sigma$ and its reversed orientation $\Sigma'$. Such manifolds are in general \textit{Quasi-Fuchsian}. Among those manifolds we have the Fuchsian example explained in \ref{example:fuchsian}
    \item\label{schcase} A disjoint union of two Schottky manifolds with boundary $S$, such that the conformal boundary of the ends are the Riemann surface $\Sigma$ and its reversed orientation $\Sigma'$. The Schottky components might be distinct, but among those examples we have a Schottky manifold with conformal boundary $\Sigma$ and the Schottky manifold obtained by reversing orientation (i.e. conjugating the group by an orientation reversing isometry of $\mathbb{H}^3$).
\end{enumerate}
We now explain a reasonable guess for which fillings between (\ref{qfcase}) and (\ref{schcase}) would produce the smallest gravitational action, and how that in turn motivates Question \ref{que:VRneg}. Given that the filling in case (\ref{qfcase}) is a more straightforward geometric construction and the gravitational action relates to the entropy of the model space, model (\ref{qfcase}) should be the one with smaller gravity action. But since the description between hyperbolic metrics in $M$ and conformal structure in $\partial M$ (via Ahlfors-Bers measurable Riemann mapping theorem, \cite{AB60}) uses \textit{markings} for the conformal structure (landing on the Teichm\"uller space of $\partial M$), we have infinitely many ways to realize cases (\ref{qfcase}) or (\ref{schcase}). Intuitively, each topological model will have assigned a partition function $\sum_{M} e^{-S_{gr}(M)}$, where $M$ is taken along the hyperbolic fillings of the given topological type. Here the value associated in each formal summand is supposed to compare the likelihood of each hyperbolic filling, or in our geometric terms, the simplicity of each filling. Hence, from an intuitive viewpoint, the comparison would be made between the leading coefficients

\begin{equation}
    \inf_{\partial M = \Sigma \text{ as in }(\ref{qfcase})}S_{gr}(M)  < \inf_{\partial M = \Sigma \text{ as in }(\ref{schcase})}S_{gr}(M)
\end{equation}
which in terms of renormalized volume correspond to

\begin{equation}\label{eq:goalcomparison}
    2\inf_{\partial M = \Sigma,\,M \text{ Schottky }}\VR(M)  < \inf_{\partial M = \Sigma,\,M \text{ Quasi-Fuchsian }}\VR(M),
\end{equation}
where the factor of $2$ is due to the $2$ components considered in case (\ref{schcase}).

Work has been done on the right side of (\ref{eq:goalcomparison}). Namely, from [\cite{VP17},Theorem 8.1] or [\cite{BBB19}, Theorem 3.11] we have that, for $M$ Quasi-Fuchsian, $\VR(M)\geq 0$. Moreover, equality holds if and only if we use the same marking for both boundary components. Then the main question reduces to:

\begin{que}\label{que:VRneg}
For a given Riemann surface $\Sigma$, does there exist a Schottky manifold $M$ so that $\partial M$ is conformal to $\Sigma$ and $\VR(M)<0$?
\end{que}

A partial answer to this question is the last component of the main result of this article. For it, we need to define \emph{extremal length} (see \cite{Ahlfors06}).

\begin{defi}\label{defi:extlength}
Given a Riemann surface $\Sigma$ and $\gamma_0\subset\Sigma$ a closed curve, we define the extremal length of $\gamma_0$ in $\Sigma$, denoted by $EL(\gamma_i,\Sigma)$, as
\begin{equation}
    EL(\gamma_0,\Sigma) := \sup_\rho \frac{\inf_{\gamma \sim\gamma_0}\ell^2_\rho(\gamma_i)}{A(\rho)}
\end{equation}
where $\gamma$ ranges over curves homotopic to $\gamma_0$, $\rho$ ranges over all metrics conformal to $\Sigma$ and $A(\rho)$ denotes the total area of $\rho$. Unless needed, we will drop the dependence on $\Sigma$ from now on.
\end{defi}

We state now our main result.

\begin{theorem}\label{thm:main}
Let $\Sigma$ be a Riemann surface of genus $g$, and let $\Gamma$ be a set of $g-1$ mutually disjoint, non-homotopic, simple closed curves of $\Sigma$ with sum of square roots of extremal lengths denoted by $L(\Sigma,\Gamma)$. Moreover, assume that each component of $\Sigma\setminus\Gamma$ has genus $1$. Then for any Schottky manifold $M$ with boundary at infinity conformal to $\Sigma$ so that the curves of $\Gamma$ are compressible in $M$, we have that
\begin{equation}
    \VR(M) \leq L(\Sigma,\Gamma)^2 + \pi(g-1)
\end{equation}

Moreover, if we further assume that $L(\Sigma,\Gamma) \leq \sqrt{\pi(g-1)}$ then we have that

\begin{equation}
    \VR(M) \leq \pi(g-1)\left(3-\frac{\pi(g-1)}{L(\Sigma,\Gamma)^2}\right)
\end{equation}
which answers positively Maldacena's question if $L(\Sigma,\Gamma)\leq \sqrt{\frac{\pi(g-1)}{3}}$.
\end{theorem}

\noindent\textbf{Remark:} Given a Riemann surface $\Sigma$, we minimize the upperbound of Theorem \ref{thm:main} by considering a Schottky manifold $M$ with boundary at infinity $\Sigma$ so that the multicurve $\Gamma$ that minimizes $L(\Sigma,\Gamma)$ is compressible. While not conclusive, this is the case where we have the best available bound for $\VR(M)$ 

\begin{equation}
    \VR(M) \leq L(\Sigma)^2 + \pi(g-1),
\end{equation}
where $L(\Sigma)$ is the least value of $L(\Sigma,\Gamma)$ across allowed multicurves $\Gamma$.

This raises the following question

\begin{que}
For a given Riemann surface $\Sigma$, which multicurve(s) realize $L(\Sigma)$?
\end{que}

\section{Preliminary results}\label{sec:prel}

The main result we will use to estimate $\VR$ is given by Schlenker in \cite{Schlenker13} in the case of Quasi-Fuchsian manifolds and by Bridgeman and Canary in \cite{BC15} for convex co-compact manifolds.

\begin{theorem}[Theorem 1.1 \cite{Schlenker13}, Proof of Theorem 1.2 \cite{BC15}]\label{thm:VRVC}
$$\VR (M) \leq \VC(M) - \frac14 L(\mu)$$
\end{theorem}

Here $\VC$ denotes the volume of the \textit{convex core} of $M$, which in turn we denote by $CC(M)$. The convex core is the smallest submanifold with convex boundary that is a homotopic retraction of $M$. The boundary of the convex core, $\partial CC(M)$, is a hyperbolic surface (with the path metric induced by $M$) whose embedding into $M$ is totally geodesic outside a closed set of disjoint complete geodesic, called the \textit{bending lamination}. Along the geodesic lamination $\partial CC(M)$ \textit{bends}, meaning that for any segment transverse to the lamination we have a well-defined bending angle. Such structure is called \textit{bending measure}, denoted by $\mu$. If we take the expected vale for the bending for a random unit segment (under the natural measure), we will obtain the \textit{total bending} of $\partial CC(M)$, also known as \textit{length of the bending lamination}, which we denote by $L(\mu)$.

A very useful construction related to the convex core is the \emph{Thurston metric} (also known as \emph{grafting metric}). The Thurston metric (see \cite[Section 2.2]{BBB19}, \cite[Chapter II.2]{EpsteinMarden06}, \cite{KamishimaTan} for more details) is a metric in the conformal class at infinity obtained by taking the hyperbolic surface $\partial CC(M)$ and adding flat regions along the bending lamination, whose thickness is given by the bending. This is easily visualize when the bending locus of the convex core is supported on disjoint closed geodesics, while the general construction and arguments are extended to general bending laminations by continuity.

With Theorem \ref{thm:VRVC} and Question \ref{que:VRneg} in mind, we will aim to prove that the term $\VC(M) - \frac14 L(\mu)$ is non-positive for certain Schottky manifolds. Hence we need a bound on $\VC$ which is comparable to $L(\mu)$, under the correct choice of compressible curves. Note that Theorem \ref{thm:VRVC} already positively answers Maldacena's question in the particular case when among all possible Schottky manifolds with a given conformal boundary there is a Fuchsian representative of the second kind (i.e. the associated group of isometries of $\mathbb{H}^3$ preserves a circle). This is because for this class of Schottky manifolds the convex core degenerates into a totally geodesic surface with boundary, so $\VC=0$ and $L(\mu)>0$.

\noindent\textbf{Remark.} In \cite{BC05} Bridgeman and Canary compare the length of the bending lamination with the inverse of injectivity radius of the Poincar\'e metric at infinity or with the inverse of the injectivity radius of the intrinsic metric of the covering of the convex hull. In \cite{BC15} the same authors use that to bound $\VR-\VC$. In our search for a bound of $\VC$ we keep \cite{BC15} results in mind, so we look for an upper bound that includes both the total bending of the convex core and short curves of the boundary.

Next, we state the tools we will use to bound volumes and related quantities.

\begin{theorem}[Hyperbolic isoperimetric inequality]\label{thm:isop} Let $B$ be a topological ball in $\HH$ with rectifiable boundary. Then $|B| < \frac12|\partial B| $.
\end{theorem}

This follows easily by verifying such inequality for round balls in $\HH$ and from the knowledge that round balls are the solution of the isoperimetric problem in $\HH$ (see for instance \cite{Schmidt48}).

For a piecewise smooth curve $\gamma$ in a Riemannian manifold $M$ (parametrized by arc-length), we can define its \textit{geodesic curvature} as $k(s)=|\gamma''(s)|$. More precisely, $k(s)$ is defined as $|\nabla_{\gamma'(t)}\gamma'(t)\vert_{t=s}|$. We can then define the \textit{total geodesic curvature} $\theta(\gamma)$ by taking the integral $\theta(\gamma)= \int_\gamma k$, and extend this definition for $C^{1,1}$ curves.

\begin{lem}\label{lem:bend}
Let $\gamma$ be a homotopically trivial $C^{1,1}$ curve in $\HH$ which is in the boundary of a convex set, and let $\theta(\gamma)$ be its total geodesic curvature. Then $\gamma$ bounds a disk $D$ of area less than $\theta(\gamma)-2\pi+\epsilon$ for all $\epsilon>0$. 
\end{lem}

\begin{proof}
Let us first prove the analogous result for the case when $\gamma$ is a geodesic polygon with $m$ vertices. Take a disk $D$ triangulated by geodesic triangles with vertices in the vertex set of $\gamma$. It is easy to see that such disk $D$ exists, as one could for instance considering $d$ to be one of the components of $\partial C_\gamma \setminus \gamma$, where $C_\gamma$ is the convex hull of $\gamma$. Note that $D$ (probably after some triangular subdivision) is a geodesic triangulation $T$ consisting on $m-2$ triangles.

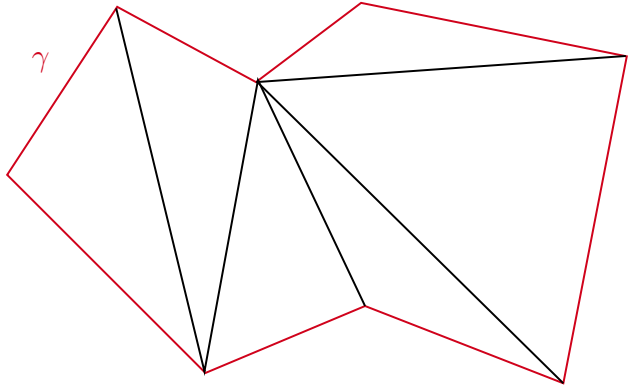
\begin{figure}[!htb]
    \centering

\tikzset{every picture/.style={line width=0.75pt}} %set default line width to 0.75pt        

\begin{tikzpicture}[x=0.75pt,y=0.75pt,yscale=-1,xscale=1]
%uncomment if require: \path (0,411.8000030517578); %set diagram left start at 0, and has height of 411.8000030517578

%Shape: Polygon [id:ds9798945582699294] 
\draw  [color={rgb, 255:red, 208; green, 2; blue, 27 }  ,draw opacity=1 ] (172,192) -- (272,292) -- (352.5,258.2) -- (452.5,297.2) -- (484.5,132.2) -- (350.5,105.2) -- (297.5,145.2) -- (227.5,107.2) -- cycle ;
%Straight Lines [id:da6092825045192722] 
\draw    (227,108) -- (271.5,291.2) -- (298.5,144.2) -- (352.5,258.2) ;

%Straight Lines [id:da5627081127584735] 
\draw    (298,145) -- (452.5,297.2) ;

%Straight Lines [id:da5227118194326517] 
\draw    (484,132) -- (297.5,145.2) ;

% Text Node
\draw (189,135) node   {$\textcolor[rgb]{0.82,0.01,0.11}{\gamma }$};

\end{tikzpicture}

    \caption{Polygonal disk with boundary $\gamma$}
    \label{fig:gamma}
\end{figure}

Given that the area of a hyperbolic triangle is $\pi-(\text{sum of its interior angles})$, then the area of $D$ is equal to $\pi(m-2) - (\text{sum of all interior angles of }T)$. Now, by angle triangular inequality, if $v$ is a vertex of $\gamma$, then $(\text{sum of angles of }T\text{ around }v) \geq (\text{interior angle of } \gamma \text{ at }v)$. Then we have

\begin{equation}
    \begin{split}
    |D| &= \pi(m-2) - (\text{sum of all interior angles of }T) \\&\leq \pi(m-2) - (\text{sum of interior angles of }\gamma) = (\text{sum of exterior angles of }\gamma) - 2\pi.
    \end{split}
\end{equation}
This bound is useful since only depends on $\gamma$ and not on the particular disk taken. For a general curve $\gamma$ take $\gamma_k$ a sequence of finer and finer polygonal approximations of $\gamma$, such that $|(\text{sum of exterior angles of }\gamma_k) - \theta(\gamma)| < \frac1k$. We can also assume that $\gamma$ and $\gamma_k$ are cobordant by an annulus of area less than $\frac1k$. Considering the union of these surfaces, given any $k>0$, $\gamma$ bounds a disk with area bounded by $(\theta(\gamma)-2\pi) + \frac2k$ 

\end{proof}

\textbf{Remark.} \textit{On the convergence of (sum of exterior angles of $\gamma_k$) to $\theta(\gamma)$}

Recall that $\nabla_{\gamma'(t)}\gamma'(t)$ can be approximated by $\frac1h(P_{-h}(\gamma'(t+h)) - \gamma'(t))$ with an error linear on $h$, where $P_{-h}$ represents the parallel transport from $\gamma(t+h)$ to $\gamma(t)$. And because the vectors $\gamma'(t), \gamma'(t+h)$ are unitary, then (up to another linear error on $h$) we can take the approximation as $\frac1h\theta_{t,t+h}$, where $\theta_{t,t+h}$ is the angle between $P_{-h}(\gamma'(t+h))$ and $ \gamma'(t)$.

Now, denoting by $\rho_{t,t+h}$ the geodesic between $\gamma(t)$ and $\gamma(t+h)$, then the angle $\theta_{t,h}$ can be calculated as the sum of the angles between $\rho_{t,t+h}$ and $\gamma$ at $\gamma(t)$ and $\gamma(t+h)$. Hence, for $\gamma_k$, the sum of its exterior angles can be rearranged as the sum of angles between the geodesic segments forming $\gamma_k$ and $\gamma$, at the vertices of $\gamma_k$ that we denote by $\gamma_k(t_i)$. This is equal to the sum of angles $\sum_i\theta_{t_i,t_{i+1}}$, which is an approximation for $\sum_ik(t_i)(t_{i+1}-t_i)$. Then, because the error was linear on the point distance, it follows that the sum of exterior angles of $\gamma_k$ converges to $\theta(\gamma)$.

\section{Bounds on $\VR$}\label{sec:bounds}

With the main ingredients introduced, we proceed to prove Theorem \ref{thm:main} in 5 steps.

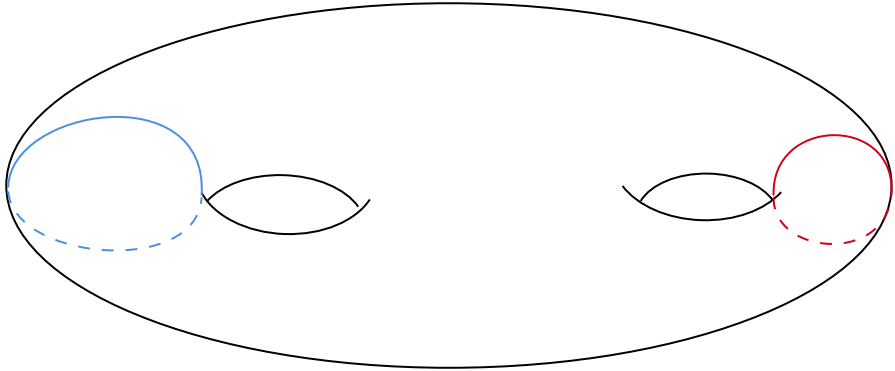
\begin{figure}[!htb]
    \centering

\tikzset{every picture/.style={line width=0.75pt}} %set default line width to 0.75pt        

\begin{tikzpicture}[x=0.75pt,y=0.75pt,yscale=-1,xscale=1]
%uncomment if require: \path (0,411.8000030517578); %set diagram left start at 0, and has height of 411.8000030517578

%Shape: Ellipse [id:dp13421436249688923] 
\draw   (100,193) .. controls (100,142.19) and (199.95,101) .. (323.25,101) .. controls (446.55,101) and (546.5,142.19) .. (546.5,193) .. controls (546.5,243.81) and (446.55,285) .. (323.25,285) .. controls (199.95,285) and (100,243.81) .. (100,193) -- cycle ;
%Shape: Arc [id:dp756507783374722] 
\draw  [draw opacity=0] (283.33,200.05) .. controls (276.25,210.87) and (259.61,218.12) .. (240.47,217.49) .. controls (220.95,216.83) and (204.57,208.18) .. (198.5,196.58) -- (241.47,187.5) -- cycle ; \draw   (283.33,200.05) .. controls (276.25,210.87) and (259.61,218.12) .. (240.47,217.49) .. controls (220.95,216.83) and (204.57,208.18) .. (198.5,196.58) ;
%Shape: Arc [id:dp3235773050209043] 
\draw  [draw opacity=0] (201.5,200.66) .. controls (208.74,192.7) and (222.93,187.46) .. (239.15,187.76) .. controls (256.42,188.09) and (271.15,194.63) .. (277.53,203.72) -- (238.66,213.26) -- cycle ; \draw   (201.5,200.66) .. controls (208.74,192.7) and (222.93,187.46) .. (239.15,187.76) .. controls (256.42,188.09) and (271.15,194.63) .. (277.53,203.72) ;
%Shape: Arc [id:dp7222600635196361] 
\draw  [draw opacity=0] (490.8,196.33) .. controls (482.74,205.3) and (467.61,211.06) .. (450.47,210.49) .. controls (433.05,209.9) and (418.14,202.95) .. (410.81,193.22) -- (451.47,180.5) -- cycle ; \draw   (490.8,196.33) .. controls (482.74,205.3) and (467.61,211.06) .. (450.47,210.49) .. controls (433.05,209.9) and (418.14,202.95) .. (410.81,193.22) ;
%Shape: Arc [id:dp0996795637495349] 
\draw  [draw opacity=0] (419.98,200.78) .. controls (425.02,192.47) and (438.43,186.69) .. (454.05,186.99) .. controls (468.68,187.27) and (481.15,192.79) .. (486.56,200.46) -- (453.64,208.55) -- cycle ; \draw   (419.98,200.78) .. controls (425.02,192.47) and (438.43,186.69) .. (454.05,186.99) .. controls (468.68,187.27) and (481.15,192.79) .. (486.56,200.46) ;
%Curve Lines [id:da4485938033906077] 
\draw [color={rgb, 255:red, 74; green, 144; blue, 226 }  ,draw opacity=1 ]   (101,194) .. controls (101.5,154) and (202.5,138) .. (198.5,198) ;

%Curve Lines [id:da45278221425173837] 
\draw [color={rgb, 255:red, 74; green, 144; blue, 226 }  ,draw opacity=1 ] [dash pattern={on 4.5pt off 4.5pt}]  (101,196) .. controls (103.5,231) and (196.5,239) .. (198.5,199) ;

%Curve Lines [id:da1699515889277754] 
\draw [color={rgb, 255:red, 208; green, 2; blue, 27 }  ,draw opacity=1 ]   (487,198) .. controls (485.5,159) and (546.5,157) .. (546.5,195) ;

%Curve Lines [id:da5187768859056527] 
\draw [color={rgb, 255:red, 208; green, 2; blue, 27 }  ,draw opacity=1 ] [dash pattern={on 4.5pt off 4.5pt}]  (487,199) .. controls (486.5,229) and (546.5,233) .. (546.5,195) ;
\end{tikzpicture}

    \caption{Handlebody for genus $g=2$}
    \label{fig:handlebody}
\end{figure}

\noindent\textbf{\emph{Step 1. First isoperimetric inequality}}

From now on let $M$ be a Schottky manifold with boundary a surface $\Sigma$ of genus $g$. Let $\Gamma$ be a collection of $g$ disjoint compressible simple curves on $\Sigma$ (meaning that each curve in $\Gamma$ bounds a disk in $M$) such that if we compress all curves in $\Gamma$ we are left with a connected $3$-ball. Our first step involves bounding $\VC(M)$ by the total length and curvature of $\Gamma$.

Denote then by $\{\gamma_i\}_{1\leq i \leq g}=\Gamma\subset \Sigma$ the collection of disjoint compressible curves such that $M$ is a ball after cutting along the compressing disks (See Figure \ref{fig:handlebody}). Moreover, assume that each $\{\gamma_i\}_{1\leq i \leq g}$ is a geodesic representative of least length in their respective homotopy class for the Thurston metric of $\partial M$.

Then, if $r$ is the projection from infinity to $\partial CC(M)$, $\{r(\gamma_i)\}_{1\leq i\leq g}$ is a collection of geodesics in $\partial CC(M)$ (with its intrinsic metric) that are compressible in $CC(M)$. Then if $D_i$ is an embedded disk in $CC(M)$ with boundary $r(\gamma_i)$, then the metric completion of $CC(M)\setminus(\cup_i D_i)$ (denoted by $X$) embeds in $\HH$. Note that $X$ is a topological ball whose boundary is made out of two copies of $D_i$ for each $1\leq i\leq g$ and a copy of $\partial CC(M) \setminus (\cup_i r(\gamma_i))$ (See Figure \ref{fig:openHB}). Hence

\begin{equation}
    |\partial X| = 4\pi(g-1) + \sum_{i=1}^{g} 2|D_i|,
\end{equation}
where we have used that the intrinsic metric of $\partial CC(M)$ is hyperbolic.

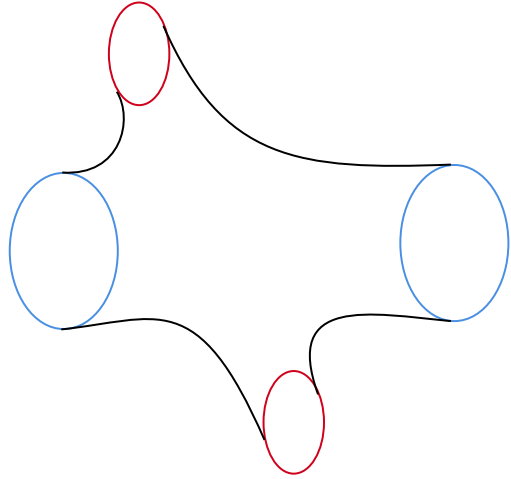
\begin{figure}
    \centering

\tikzset{every picture/.style={line width=0.75pt}} %set default line width to 0.75pt        

\begin{tikzpicture}[x=0.75pt,y=0.75pt,yscale=-1,xscale=1]
%uncomment if require: \path (0,411.8000030517578); %set diagram left start at 0, and has height of 411.8000030517578

%Shape: Ellipse [id:dp7295863423306591] 
\draw  [color={rgb, 255:red, 208; green, 2; blue, 27 }  ,draw opacity=1 ] (245,77.9) .. controls (245,63.6) and (251.83,52) .. (260.25,52) .. controls (268.67,52) and (275.5,63.6) .. (275.5,77.9) .. controls (275.5,92.2) and (268.67,103.8) .. (260.25,103.8) .. controls (251.83,103.8) and (245,92.2) .. (245,77.9) -- cycle ;
%Shape: Ellipse [id:dp04561016258676287] 
\draw  [color={rgb, 255:red, 208; green, 2; blue, 27 }  ,draw opacity=1 ] (323,263.9) .. controls (323,249.6) and (329.83,238) .. (338.25,238) .. controls (346.67,238) and (353.5,249.6) .. (353.5,263.9) .. controls (353.5,278.2) and (346.67,289.8) .. (338.25,289.8) .. controls (329.83,289.8) and (323,278.2) .. (323,263.9) -- cycle ;
%Shape: Ellipse [id:dp9173160522643743] 
\draw  [color={rgb, 255:red, 74; green, 144; blue, 226 }  ,draw opacity=1 ] (195,177.4) .. controls (195,155.64) and (207.2,138) .. (222.25,138) .. controls (237.3,138) and (249.5,155.64) .. (249.5,177.4) .. controls (249.5,199.16) and (237.3,216.8) .. (222.25,216.8) .. controls (207.2,216.8) and (195,199.16) .. (195,177.4) -- cycle ;
%Shape: Ellipse [id:dp4921481172992467] 
\draw  [color={rgb, 255:red, 74; green, 144; blue, 226 }  ,draw opacity=1 ] (392,173.4) .. controls (392,151.64) and (404.2,134) .. (419.25,134) .. controls (434.3,134) and (446.5,151.64) .. (446.5,173.4) .. controls (446.5,195.16) and (434.3,212.8) .. (419.25,212.8) .. controls (404.2,212.8) and (392,195.16) .. (392,173.4) -- cycle ;
%Curve Lines [id:da07322562960238443] 
\draw    (272.5,63.8) .. controls (302.5,131.8) and (339.5,136.8) .. (417.5,133.8) ;

%Curve Lines [id:da09514499574677981] 
\draw    (249,97) .. controls (257.5,110.8) and (251.5,139.8) .. (221.5,137.8) ;

%Curve Lines [id:da4325789297592817] 
\draw    (221,217) .. controls (269.5,211.8) and (289.5,194.8) .. (323.5,272.8) ;

%Curve Lines [id:da5352383111651359] 
\draw    (350.5,249.8) .. controls (332.5,199.8) and (376.5,208.8) .. (417.5,212.8) ;

\end{tikzpicture}

    \caption{$X\subset\HH$ for genus $g=2$}
    \label{fig:openHB}
\end{figure}

By the hyperbolic isoperimetric inequality (Theorem \ref{thm:isop}) we have

\begin{equation}\label{eq:boundvol}
    |X| < 2\pi(g-1) + \sum_{i=1}^{g} |D_i|,
\end{equation}
which gives a bound on $\VC(M)$ in terms of $|D_i|$. In order to use Lemma \ref{lem:bend} to bound the area of $|D_i|$, we need to bound the bending along of $r(\gamma_i)$. The key observation is that we can approximate $r(\gamma_i)$ by curves with total geodesic curvature bounded by $\ell_T(\gamma_i)$, where $\ell_T$ denotes the length in the Thurston metric. More precisely, we will show in the next step that those total curvatures are  bounded by the length of $\gamma_i$ in the flat pieces of the Thurston metric.

\noindent\textbf{\emph{Step 2. Bounding the bending of the projection of a Thurston geodesic}}

Denote by $\Sigma_\delta$ the set of points in the complement of $CC(M)$ at distance $\delta$ of $CC(M)$. $\Sigma_\delta$ is a $C^{1,1}$ surface whose geometry can be described in a similar to how we described the Thurston metric. In fact, it can be obtained by the construction of Epstein (see \cite{Epstein} and \cite[Section II.2.2]{EpsteinMarden06}) while considering a constant conformal multiple of the Thurston metric. In the complement of the bending lamination, the equidistant set in $\Sigma_\delta$ corresponding to $CC(M)$ is a totally umbilic surface with principal curvatures equal to $\tanh{\delta}$. On the other hand, the equidistant set in $\Sigma_\delta$ corresponding to the bending locus are intrinsically flat strips with principal curvatures $\tanh{\delta}, 1/\tanh{\delta}$. The union of these flat strips attach to the umbilic components to make the closed $C^{1,1}$ surface $\Sigma_\delta$, where the thickness of the flat strips are given by the bending times $\sinh{\delta}$. In fact, by \cite{Epstein} the Thurston metric can be obtained by taking the limit as $\delta\rightarrow\infty$ of the intrinsic metric of $\Sigma_\delta$ re-scaled by a factor of $4e^{-2\delta}$, after identifying the surfaces $\Sigma_\delta$ by the normal exponential map, although we will not use this fact.

Under the identification given by the normal exponential map, denote by $\gamma_{i,\delta}$ the curve that corresponds to $\gamma_i$ in $\Sigma_\delta$. We have that $\gamma_{i,\delta}\rightarrow r(\gamma_i)$ uniformly as $\delta\rightarrow 0$ and that $\gamma_{i,\delta}$ is geodesic in the convex surface $\Sigma_\delta$. Then, after parametrizing by arc length, the geodesic curvature of $\gamma_{i,\delta}$ at any point given by $\II_\delta(\gamma_{i,\delta}', \gamma_{i,\delta}')$, where $\II_\delta$ is the second fundamental form of $\Sigma_\delta$. This means that the geodesic curvature of $\gamma_{i,\delta}$ is $\tanh{\delta}$ at the umbilic components, and equal to $\tanh{\delta}\cos{\alpha} + \frac{1}{\tanh{\delta}}\sin(\alpha)$ in the flat strips, where $\alpha$ is the angle of $\gamma_{i,\delta}'$ with the horizontal direction of the strip. Since $\gamma$ represents a compressible curve, it is in particular transverse to the bending lamination. So it follows that the intersection of $\gamma_{i,\delta}$ with the attaching locus has measure $0$ (even though it could be as complicated as a Cantor set), and we have then that
\[\theta(\gamma_{i,\delta}) \leq (\tanh{\delta}\sinh{\delta} + \cosh{\delta}\mu(\gamma_{i,\delta})) + \tanh{\delta}.\ell(\gamma_{i,\delta}),
\]
where $\mu(\gamma_{i,\delta})$ denotes the bending measure of $\gamma_{i,\delta}$ and $\ell(\gamma_{i,\delta})$ denotes the length of $\gamma_{i,\delta}$ in $\partial CC(M)$. Hence as $\delta\rightarrow 0$ then $\limsup \theta(\gamma_{i,\delta})$ is bounded by the bending along $\gamma_i$, which in turn is bounded by $\ell_T(\gamma_{i,\delta})$, as $\mu(\gamma_{i,\delta})$ only accounts for the length in the flat pieces of the Thurston metric.

Applying then Lemma \ref{lem:bend} we conclude that

\begin{equation}
    |D_i| \leq \ell_T(\gamma_i) - 2\pi+\epsilon,
\end{equation}
for certain disks $D_i$ that in principle depend on $\epsilon$. Nevertheless, we can replace the upper bound $\ell_T(\gamma_i) - 2\pi$ in equation (\ref{eq:boundvol}) to obtain

\begin{equation}\label{eq:isoperimetricballineq}
    |X| < 2\pi(g-1) + \sum_{i=1}^{g} (\ell_T(\gamma_i) - 2\pi),
\end{equation}
since $\epsilon$ is arbitrary small. Observe that we also recover the well-known fact $\ell_T(\gamma_i)>2\pi$ (see \cite[Theorem 2]{BonahonOtal}).

\noindent\textbf{\emph{Step 3. Main isoperimetric inequality}}

The next natural step is to compare the right hand side of inequality (\ref{eq:isoperimetricballineq})  with $\frac14 L(\mu)$, but before doing so we will strengthen (\ref{eq:isoperimetricballineq}) by applying it to finite covers of the manifold and divide the inequalities by the order of the cover. More precisely, we apply the isoperimetric inequality (\ref{eq:isoperimetricballineq})  to the union of many isometric copies of $X_0$ and we will keep the simplest bound (inequality (\ref{eq:2ndbound})). The main conclusion of this step is inequality (\ref{eq:VrvsEL}) applied to multicurves $\Gamma$ in $\Sigma$ as stated in Theorem \ref{thm:main}.

Denote by $D^{\pm}_1,\ldots,D^{\pm}_g$ the disks bounding each of the $g$ compressible curves, where the notation $\pm$ is used to differentiate between the two boundary regions (per disk) on the lift $X$ of the convex core (see the color labelling of Figure \ref{fig:openHB}). Now we can take the union of $X=X_0$ with its translations by the isometries that identify $D^+_g$ with $D^-_g$, and we denote such union by $X_1$ (see Figure \ref{fig:betterX1}). For this region we can apply again the isoperimetric inequality, which gives us a better estimate on the volume. Let then $X_n$ denote the solid defined inductively as $X_{n-1}$ union the adjacent isometric copies of $X_0$ through the faces corresponding to $D^\pm_g$. Then $X_n$ is the union of $2n+1$ isometric copies of $X_0$, while $\partial X_n$ is the union of $2n+1$ $2g$-holed spheres (with area $4\pi(g-1)$ each), (2n+1) copies of each $D^\pm_i$ ($1\leq i\leq g-1$), and one isometric copy of each $D^\pm_g$. Applying then the isoperimetric inequality to $X_n$  (and recalling that $|D^\pm_i| = |D_i|$) we obtain

\begin{figure}[!htb]
    \centering

\tikzset{every picture/.style={line width=0.75pt}} %set default line width to 0.75pt        

\tikzset{every picture/.style={line width=0.75pt}} %set default line width to 0.75pt        

\begin{tikzpicture}[x=0.75pt,y=0.75pt,yscale=-1,xscale=1]
%uncomment if require: \path (0,411.8000030517578); %set diagram left start at 0, and has height of 411.8000030517578

%Shape: Ellipse [id:dp7295863423306591] 
\draw  [color={rgb, 255:red, 208; green, 2; blue, 27 }  ,draw opacity=1 ] (256.29,122.06) .. controls (256.29,112.08) and (261.11,104) .. (267.05,104) .. controls (273,104) and (277.81,112.08) .. (277.81,122.06) .. controls (277.81,132.03) and (273,140.12) .. (267.05,140.12) .. controls (261.11,140.12) and (256.29,132.03) .. (256.29,122.06) -- cycle ;
%Shape: Ellipse [id:dp04561016258676287] 
\draw  [color={rgb, 255:red, 208; green, 2; blue, 27 }  ,draw opacity=1 ] (311.34,251.74) .. controls (311.34,241.77) and (316.16,233.68) .. (322.1,233.68) .. controls (328.05,233.68) and (332.86,241.77) .. (332.86,251.74) .. controls (332.86,261.72) and (328.05,269.8) .. (322.1,269.8) .. controls (316.16,269.8) and (311.34,261.72) .. (311.34,251.74) -- cycle ;
%Shape: Ellipse [id:dp9173160522643743] 
\draw  [color={rgb, 255:red, 74; green, 144; blue, 226 }  ,draw opacity=1 ][dash pattern={on 4.5pt off 4.5pt}] (221,191.43) .. controls (221,176.26) and (229.61,163.96) .. (240.23,163.96) .. controls (250.85,163.96) and (259.46,176.26) .. (259.46,191.43) .. controls (259.46,206.6) and (250.85,218.9) .. (240.23,218.9) .. controls (229.61,218.9) and (221,206.6) .. (221,191.43) -- cycle ;
%Shape: Ellipse [id:dp4921481172992467] 
\draw  [color={rgb, 255:red, 74; green, 144; blue, 226 }  ,draw opacity=1 ][dash pattern={on 4.5pt off 4.5pt}] (360.04,188.64) .. controls (360.04,173.47) and (368.65,161.17) .. (379.27,161.17) .. controls (389.89,161.17) and (398.5,173.47) .. (398.5,188.64) .. controls (398.5,203.81) and (389.89,216.11) .. (379.27,216.11) .. controls (368.65,216.11) and (360.04,203.81) .. (360.04,188.64) -- cycle ;
%Curve Lines [id:da07322562960238443] 
\draw    (275.7,112.23) .. controls (296.87,159.64) and (322.98,163.12) .. (378.03,161.03) ;

%Curve Lines [id:da09514499574677981] 
\draw    (259.11,135.38) .. controls (265.11,145) and (260.88,165.22) .. (239.7,163.82) ;

%Curve Lines [id:da4325789297592817] 
\draw    (239.35,219.04) .. controls (273.58,215.42) and (287.69,203.56) .. (311.69,257.95) ;

%Curve Lines [id:da5352383111651359] 
\draw    (330.75,241.91) .. controls (318.04,207.05) and (349.1,213.32) .. (378.03,216.11) ;

%Shape: Ellipse [id:dp5552662782078707] 
\draw  [color={rgb, 255:red, 208; green, 2; blue, 27 }  ,draw opacity=1 ] (397.29,120.06) .. controls (397.29,110.08) and (402.11,102) .. (408.05,102) .. controls (414,102) and (418.81,110.08) .. (418.81,120.06) .. controls (418.81,130.03) and (414,138.12) .. (408.05,138.12) .. controls (402.11,138.12) and (397.29,130.03) .. (397.29,120.06) -- cycle ;
%Shape: Ellipse [id:dp7999094294474034] 
\draw  [color={rgb, 255:red, 208; green, 2; blue, 27 }  ,draw opacity=1 ] (452.34,249.74) .. controls (452.34,239.77) and (457.16,231.68) .. (463.1,231.68) .. controls (469.05,231.68) and (473.86,239.77) .. (473.86,249.74) .. controls (473.86,259.72) and (469.05,267.8) .. (463.1,267.8) .. controls (457.16,267.8) and (452.34,259.72) .. (452.34,249.74) -- cycle ;
%Shape: Ellipse [id:dp986068865607945] 
\draw  [color={rgb, 255:red, 74; green, 144; blue, 226 }  ,draw opacity=1 ] (501.04,186.64) .. controls (501.04,171.47) and (509.65,159.17) .. (520.27,159.17) .. controls (530.89,159.17) and (539.5,171.47) .. (539.5,186.64) .. controls (539.5,201.81) and (530.89,214.11) .. (520.27,214.11) .. controls (509.65,214.11) and (501.04,201.81) .. (501.04,186.64) -- cycle ;
%Curve Lines [id:da5444973756292707] 
\draw    (416.7,110.23) .. controls (437.87,157.64) and (463.98,161.12) .. (519.03,159.03) ;

%Curve Lines [id:da23616052355689154] 
\draw    (400.11,133.38) .. controls (406.11,143) and (401.88,163.22) .. (380.7,161.82) ;

%Curve Lines [id:da1159043455695925] 
\draw    (380.35,217.04) .. controls (414.58,213.42) and (428.69,201.56) .. (452.69,255.95) ;

%Curve Lines [id:da5594128683677312] 
\draw    (471.75,239.91) .. controls (459.04,205.05) and (490.1,211.32) .. (519.03,214.11) ;

%Shape: Ellipse [id:dp5188013928482115] 
\draw  [color={rgb, 255:red, 208; green, 2; blue, 27 }  ,draw opacity=1 ] (119.29,125.06) .. controls (119.29,115.08) and (124.11,107) .. (130.05,107) .. controls (136,107) and (140.81,115.08) .. (140.81,125.06) .. controls (140.81,135.03) and (136,143.12) .. (130.05,143.12) .. controls (124.11,143.12) and (119.29,135.03) .. (119.29,125.06) -- cycle ;
%Shape: Ellipse [id:dp8329406633969406] 
\draw  [color={rgb, 255:red, 208; green, 2; blue, 27 }  ,draw opacity=1 ] (174.34,254.74) .. controls (174.34,244.77) and (179.16,236.68) .. (185.1,236.68) .. controls (191.05,236.68) and (195.86,244.77) .. (195.86,254.74) .. controls (195.86,264.72) and (191.05,272.8) .. (185.1,272.8) .. controls (179.16,272.8) and (174.34,264.72) .. (174.34,254.74) -- cycle ;
%Shape: Ellipse [id:dp8353567972638577] 
\draw  [color={rgb, 255:red, 74; green, 144; blue, 226 }  ,draw opacity=1 ] (84,194.43) .. controls (84,179.26) and (92.61,166.96) .. (103.23,166.96) .. controls (113.85,166.96) and (122.46,179.26) .. (122.46,194.43) .. controls (122.46,209.6) and (113.85,221.9) .. (103.23,221.9) .. controls (92.61,221.9) and (84,209.6) .. (84,194.43) -- cycle ;
%Curve Lines [id:da31768787443828983] 
\draw    (138.7,115.23) .. controls (159.87,162.64) and (185.98,166.12) .. (241.03,164.03) ;

%Curve Lines [id:da8202846970423802] 
\draw    (122.11,138.38) .. controls (128.11,148) and (123.88,168.22) .. (102.7,166.82) ;

%Curve Lines [id:da5487750604816559] 
\draw    (102.35,222.04) .. controls (136.58,218.42) and (150.69,206.56) .. (174.69,260.95) ;

%Curve Lines [id:da06243188169244229] 
\draw    (193.75,244.91) .. controls (181.04,210.05) and (212.1,216.32) .. (241.03,219.11) ;

\end{tikzpicture}

    \caption{$X_1$ when following adjacency only through $D^\pm_g$ for $g=2$}
    \label{fig:betterX1}
\end{figure}

\begin{equation}
    (2n+1)|X_0| = |X_n| \leq \frac12|\partial X_n| = \frac12 \left( (2n+1)4\pi(g-1) + \sum_{i=1}^{g-1}(2n+1)2|D_i| +  2|D_g| \right).
\end{equation}
Dividing by $2n+1$ and sending $n$ to infinity we obtain

\begin{equation}\label{eq:2ndbound}
    |X_0| \leq 2\pi(g-1) + \sum_{i=1}^{g-1} |D_i|.
\end{equation}
Observe that this in particular establishes an isoperimetric inequality for a solid torus, as the term $2\pi(g-1) + \sum_{i=1}^{g-1} |D_i|$ corresponds to the area of the boundary of the solid torus obtained by cutting along $\lbrace D_i\rbrace_{1\leq i\leq g-1}$.

Finally, we apply equation (\ref{eq:2ndbound}) to bound $\VC$ and subsequently $\VR$. Denote the union of $\gamma_i$ by $\Gamma$, then the bound for $\VC$ using equation (\ref{eq:2ndbound}) and Lemma \ref{lem:bend}  is given by

\begin{equation}\label{eq:VcvsEL}
    \VC(M) \leq 2\pi(g-1) + \ell_T(\Gamma) -2\pi(g-1) = \ell_T(\Gamma),
\end{equation}
which combined with Theorem \ref{thm:VRVC} gives us a bound for $\VR$

\begin{equation}\label{eq:VrvsEL}
    \VR(M) \leq \ell_T(\Gamma) - \frac14L(\mu)
\end{equation}

Note that the improvement (\ref{eq:2ndbound}) on the isoperimetric inequality has substituted the upper bound (\ref{eq:isoperimetricballineq}) (which works for a compressible multicurve $\Gamma$ so that $\Sigma\setminus\Gamma$ is connected with genus $0$) with an upper bound which works for a compressible multicurve $\Gamma$ so that $\Sigma\setminus\Gamma$ is connected with genus $1$. The last improvement we will do in this step will be to establish (\ref{eq:VcvsEL}), (\ref{eq:VrvsEL}) for compressible multicurves $\Gamma$ so that $\Sigma\setminus\Gamma$ is potentially disconnected but each component of $\Sigma\setminus\Gamma$ has genus $1$.

For such $\Gamma$ we argue as follows. Observe that if we add compressing disks of $\Gamma$ to $\Sigma\setminus\Gamma$ we obtain the boundary of solid tori $T_1,\ldots, T_k$, whose union is the convex core of $M$. To the collection of curves $\Gamma$ add one simple closed compressible curve for each component of $\Sigma\setminus\Gamma$ to obtain a collection of compressible curves $\Gamma'\supset\Gamma$. Then for $\Gamma'$ we have that each component of $\Sigma\setminus\Gamma'$ has genus $0$, and by adding compressing disks to them each component bounds a $3$-ball. Then we can apply the present step to these balls to obtain the volume bound (\ref{eq:2ndbound}) for each solid torus $T_1,\ldots, T_k$ obtained from $\Gamma$, where (\ref{eq:2ndbound}) would be
\[|T_j|\leq \frac12|T_j\cap (\Sigma\setminus\Gamma)| + \sum_{\text{ D compressing disk adjacent to } T_j} \frac12|D|
\]

Each curve in $\gamma\in\Gamma$ will contribute the area of a compressing disk once in each inequality associated to an adjacent component to $T_1,\ldots, T_k$, and it appears twice if $\gamma$ is adjacent to the same component from both sides. Finally, by adding all these inequalities and recalling that the volume of the solid tori adds to the volume of the convex core, the areas of $T_1\cap(\Sigma\setminus\Gamma),\ldots, T_k\cap(\Sigma\setminus\Gamma)$ add to $4\pi(g-1)$, and each compressing disk for $\gamma\in\Gamma$ appears exactly twice along the inequalities, we can conclude 

\begin{equation}
    \VC(M) \leq 2\pi(g-1) + \sum_{\gamma\in\Gamma} |D_\gamma|
\end{equation}
where $D_\gamma$ is a compressing disk of $\gamma$ in $CC(M)$. By employing Lemma \ref{lem:bend} as before we see that (\ref{eq:VcvsEL}), (\ref{eq:VrvsEL}) also hold whenever each component of $\Sigma\setminus\Gamma$ has genus $1$.

\noindent\textbf{\emph{Step 4. Extremal length}}

With the simplification of Step 3 done, in the next step we will bound $\ell_T(\Gamma)$ by a constant (depending only on the genus $g$ and the conformal structure of $\Sigma$) times $\sqrt{L(\mu) + 4\pi(g-1)}$, which is the square root of the area of the Thurston metric. The exponents are the natural ones to consider so that the inequality is scale invariant. Motivated by this and of desire to describe $\ell_T$ in terms that depend only on the Riemann surface $\Sigma$, we appeal to the concept of \textit{extremal length} (Definition \ref{defi:extlength})

By considering the Thurston metric, which has area $L(\mu) + 4\pi(g-1)$, it is easy to see from Definition~\ref{defi:extlength} that $\ell_T(\gamma_i) \leq \sqrt{EL(\gamma_i)}\sqrt{(L(\mu) + 4\pi(g-1))}$.

For a closed surface $S$ denote by $\mathcal{C}_{g-1}(S)$ the collection of (unordered) $(g-1)$-tuples of distinct non-trivial homotopy classes in $S$ that can be simultaneously represented by pairwise disjoint simple closed curves so that each component of their complement has genus $1$. Moreover, if $\Sigma$ is a Riemann surface with topological type $S$ then we also define

\[
    L(\Sigma,\Gamma) =  \sum_{i=1}^{g-1}\sqrt{EL(\gamma_i)} ,\, \{\gamma_i\}_{1\leq i\leq g-1}=\Gamma\in\mathcal{C}_{g-1}(S)
\]

\[
    L(\Sigma) = \min_{\Gamma\in\mathcal{C}_{g-1}(S)}\{ L(\Sigma,\Gamma) | \Gamma\in\mathcal{C}_{g-1}(S)\}
\]
Observe that these functions are defined over $\MM(S)\times C_{g-1}(S)$ and $\MM(S)$, where $\MM(S)$ denotes the moduli space of Riemann surfaces of topological type $S$.

Then, by Equation (\ref{eq:VrvsEL}) and the definition of $L(\Sigma, \Gamma)$, the bound for $\VR$ becomes:

\begin{equation}\label{eq:polybound}
    \begin{split}
    \VR(M) &\leq L(\Sigma,\Gamma)\sqrt{4\pi(g-1) + L(\mu)}-\frac14 L(\mu)\\
    &= L(\Sigma,\Gamma)\sqrt{4\pi(g-1) + L(\mu)} -\frac14(4\pi(g-1) + L(\mu)) + \pi(g-1)
    \end{split}
\end{equation}
as long as $\Gamma$ is compressible in $M$.

Since the quadratic polynomial $P_K(x)=Kx-\frac{x^2}{4}$ attains its unique maximum at $x=2K$ with value $K^2$, then we have

\begin{equation}\label{eq:thm0}
    \VR \leq L(\Sigma,\Gamma)^2 + \pi(g-1)
\end{equation}
and in particular there exists $M$ with $\partial M$ conformal to $\Sigma$ so that

\begin{equation}\label{eq:thm1}
    \VR \leq L(\Sigma)^2 + \pi(g-1)
\end{equation}

Note that since $\MM(S)$ is non-compact, it is of interest for inequality (\ref{eq:thm1}) to known whether $L:\MM(S) \rightarrow \RR$ has bounded image or not. That is the purpose of the next lemma.

\begin{lem}\label{lem:bimage}
$L:\MM(S) \rightarrow \RR$ has bounded image.
\end{lem}
\begin{proof}
We will detail a well-known argument that in particular obtains an explicit bound for such constant. By \cite[Corollary 3]{Maskit} we can bound the extremal length of a short pants decomposition with respect to the hyperbolic metric. Namely, from \cite[Remark 5.2.5 (i)]{Buser} it follows that any Riemann surface $\Sigma$ of genus $g$ has a pants decomposition so that the hyperbolic length of each curve is less than or equal than $21(g-1)$. \cite[Corollary 3]{Maskit} establishes that extremal length of a simple closed curve is bounded above by $\frac\ell2 e^{\ell/2}$, where $\ell$ denotes the hyperbolic length of the geodesic representative in the homotopy class of the curve. Combining these results and after some basic calculus we can bound $L(\Sigma)$ by $\sqrt{\frac{21}{2}}(g-1)^{3/2}e^{\frac{21(g-1)}{4}}$.
\end{proof}

\noindent\textbf{\emph{Step 5. Stronger inequalities for ``small" extremal length.}}

Recall that right after equation (\ref{eq:isoperimetricballineq}) we observed that the length of each individual compressing curve $\ell_T(\gamma_i)$ was greater than $2\pi$, so in particular $\sqrt{EL(\gamma_i)}\sqrt{(L(\mu) + 4\pi(g-1))} \geq \ell_T(\gamma_i) > 2\pi$. Adding all these inequalities we have

\begin{equation}\label{eq:known}
    L(\Sigma,\Gamma)\sqrt{L(\mu) + 4\pi(g-1)} > 2\pi(g-1).
\end{equation}
This in particular says that the maximization process for the quadratic polynomial $P_K(x)=Kx-\frac{x^2}{4}$ is actually on a restricted domain. On this final step we will take advantage of that phenomenon to improve our bound on $\VR$ for appropriately restricted conditions.

Consider then general genus $g$ and add the assumption $L(\Sigma,\Gamma) \leq \sqrt{\pi(g-1)} $. Applying this to (\ref{eq:known}) we obtain  $\sqrt{4\pi(g-1) + L(\mu)}> \frac{2\pi(g-1)}{L(\Sigma,\Gamma)}\geq 2L(\Sigma,\Gamma)$. This inequality says then that the range of $x=\sqrt{4\pi(g-1)+L(\mu)}$ for the quadratic polynomial $L(\Sigma,\Gamma)x-\frac14 x^2$ does not contain its maximum critical value $2L(\Sigma,\Gamma)$, thus we can replace $x=\frac{2\pi(g-1)}{L(\Sigma,\Gamma)}$ in (\ref{eq:polybound}) to obtain the upper bound 

\begin{equation}\label{eq:thm2}
    \begin{split}
    \VR (M) &\leq 2\pi(g-1) - \frac14\left( \frac{4\pi^2(g-1)^2}{L(\Sigma,\Gamma)^2}\right) + \pi(g-1)\\
    &= \pi(g-1)\left(3-\frac{\pi(g-1)}{L(\Sigma,\Gamma)^2}\right)
    \end{split}
\end{equation}
so then $\VR(M)$ is negative if $L(\Sigma,\Gamma) \leq \sqrt{\frac{\pi(g-1)}{3}}$. Observe that even in $g=2$ this restriction is stronger than $\frac{2}{\sqrt{3}}$ bound on the systolic ratio.

\textbf{Proof of Theorem \ref{thm:main}}: The main theorem follows from inequalities (\ref{eq:thm0}), (\ref{eq:thm1}) and (\ref{eq:thm2}).
\qed

\begin{remark}
Although on Step 3 we proceeded by using (\ref{eq:2ndbound}), there are other similar inequalities we could have chosen instead. Let us justify of our particular selection.

In a similar way to $X_n$ define $Y_n$ inductively as the solid defined by attaching to $Y_{n-1}$ \emph{all} the adjacent isometric copies of $X_0$ (see Figure \ref{fig:X1}, $Y_0=X_0$). Then if $d_n$ denotes the number of copies of $D^{\pm}_i$ (which is not hard to see that stays the same for all $1\leq i\leq g$ and $\pm$) and $x_n$ the number of copies of $X_0$ in $Y_n$, we have that they satisfy the following recurrence formulas

\begin{equation}
    \begin{split}
        x_{n+1} &= x_n + 2g.d_n\\
        d_{n+1} &= (2g-1)d_n
    \end{split}
\end{equation}
so then $d_n=(2g-1)^n,\,x_n=\frac{g}{g-1}((2g-1)^n-1) + 1$.

\begin{figure}[!htb]
    \centering

\tikzset{every picture/.style={line width=0.75pt}} %set default line width to 0.75pt        

\begin{tikzpicture}[x=0.75pt,y=0.75pt,yscale=-1,xscale=1]
%uncomment if require: \path (0,411.8000030517578); %set diagram left start at 0, and has height of 411.8000030517578

%Shape: Ellipse [id:dp7295863423306591] 
\draw  [color={rgb, 255:red, 208; green, 2; blue, 27 }  ,draw opacity=1 ][dash pattern={on 4.5pt off 4.5pt}] (256.29,122.06) .. controls (256.29,112.08) and (261.11,104) .. (267.05,104) .. controls (273,104) and (277.81,112.08) .. (277.81,122.06) .. controls (277.81,132.03) and (273,140.12) .. (267.05,140.12) .. controls (261.11,140.12) and (256.29,132.03) .. (256.29,122.06) -- cycle ;
%Shape: Ellipse [id:dp04561016258676287] 
\draw  [color={rgb, 255:red, 208; green, 2; blue, 27 }  ,draw opacity=1 ][dash pattern={on 4.5pt off 4.5pt}] (311.34,251.74) .. controls (311.34,241.77) and (316.16,233.68) .. (322.1,233.68) .. controls (328.05,233.68) and (332.86,241.77) .. (332.86,251.74) .. controls (332.86,261.72) and (328.05,269.8) .. (322.1,269.8) .. controls (316.16,269.8) and (311.34,261.72) .. (311.34,251.74) -- cycle ;
%Shape: Ellipse [id:dp9173160522643743] 
\draw  [color={rgb, 255:red, 74; green, 144; blue, 226 }  ,draw opacity=1 ][dash pattern={on 4.5pt off 4.5pt}] (221,191.43) .. controls (221,176.26) and (229.61,163.96) .. (240.23,163.96) .. controls (250.85,163.96) and (259.46,176.26) .. (259.46,191.43) .. controls (259.46,206.6) and (250.85,218.9) .. (240.23,218.9) .. controls (229.61,218.9) and (221,206.6) .. (221,191.43) -- cycle ;
%Shape: Ellipse [id:dp4921481172992467] 
\draw  [color={rgb, 255:red, 74; green, 144; blue, 226 }  ,draw opacity=1 ][dash pattern={on 4.5pt off 4.5pt}] (360.04,188.64) .. controls (360.04,173.47) and (368.65,161.17) .. (379.27,161.17) .. controls (389.89,161.17) and (398.5,173.47) .. (398.5,188.64) .. controls (398.5,203.81) and (389.89,216.11) .. (379.27,216.11) .. controls (368.65,216.11) and (360.04,203.81) .. (360.04,188.64) -- cycle ;
%Curve Lines [id:da07322562960238443] 
\draw    (275.7,112.23) .. controls (296.87,159.64) and (322.98,163.12) .. (378.03,161.03) ;

%Curve Lines [id:da09514499574677981] 
\draw    (259.11,135.38) .. controls (265.11,145) and (260.88,165.22) .. (239.7,163.82) ;

%Curve Lines [id:da4325789297592817] 
\draw    (239.35,219.04) .. controls (273.58,215.42) and (287.69,203.56) .. (311.69,257.95) ;

%Curve Lines [id:da5352383111651359] 
\draw    (330.75,241.91) .. controls (318.04,207.05) and (349.1,213.32) .. (378.03,216.11) ;

%Shape: Ellipse [id:dp5552662782078707] 
\draw  [color={rgb, 255:red, 208; green, 2; blue, 27 }  ,draw opacity=1 ] (397.29,120.06) .. controls (397.29,110.08) and (402.11,102) .. (408.05,102) .. controls (414,102) and (418.81,110.08) .. (418.81,120.06) .. controls (418.81,130.03) and (414,138.12) .. (408.05,138.12) .. controls (402.11,138.12) and (397.29,130.03) .. (397.29,120.06) -- cycle ;
%Shape: Ellipse [id:dp7999094294474034] 
\draw  [color={rgb, 255:red, 208; green, 2; blue, 27 }  ,draw opacity=1 ] (452.34,249.74) .. controls (452.34,239.77) and (457.16,231.68) .. (463.1,231.68) .. controls (469.05,231.68) and (473.86,239.77) .. (473.86,249.74) .. controls (473.86,259.72) and (469.05,267.8) .. (463.1,267.8) .. controls (457.16,267.8) and (452.34,259.72) .. (452.34,249.74) -- cycle ;
%Shape: Ellipse [id:dp986068865607945] 
\draw  [color={rgb, 255:red, 74; green, 144; blue, 226 }  ,draw opacity=1 ] (501.04,186.64) .. controls (501.04,171.47) and (509.65,159.17) .. (520.27,159.17) .. controls (530.89,159.17) and (539.5,171.47) .. (539.5,186.64) .. controls (539.5,201.81) and (530.89,214.11) .. (520.27,214.11) .. controls (509.65,214.11) and (501.04,201.81) .. (501.04,186.64) -- cycle ;
%Curve Lines [id:da5444973756292707] 
\draw    (416.7,110.23) .. controls (437.87,157.64) and (463.98,161.12) .. (519.03,159.03) ;

%Curve Lines [id:da23616052355689154] 
\draw    (400.11,133.38) .. controls (406.11,143) and (401.88,163.22) .. (380.7,161.82) ;

%Curve Lines [id:da1159043455695925] 
\draw    (380.35,217.04) .. controls (414.58,213.42) and (428.69,201.56) .. (452.69,255.95) ;

%Curve Lines [id:da5594128683677312] 
\draw    (471.75,239.91) .. controls (459.04,205.05) and (490.1,211.32) .. (519.03,214.11) ;

%Shape: Ellipse [id:dp5188013928482115] 
\draw  [color={rgb, 255:red, 208; green, 2; blue, 27 }  ,draw opacity=1 ] (119.29,125.06) .. controls (119.29,115.08) and (124.11,107) .. (130.05,107) .. controls (136,107) and (140.81,115.08) .. (140.81,125.06) .. controls (140.81,135.03) and (136,143.12) .. (130.05,143.12) .. controls (124.11,143.12) and (119.29,135.03) .. (119.29,125.06) -- cycle ;
%Shape: Ellipse [id:dp8329406633969406] 
\draw  [color={rgb, 255:red, 208; green, 2; blue, 27 }  ,draw opacity=1 ] (174.34,254.74) .. controls (174.34,244.77) and (179.16,236.68) .. (185.1,236.68) .. controls (191.05,236.68) and (195.86,244.77) .. (195.86,254.74) .. controls (195.86,264.72) and (191.05,272.8) .. (185.1,272.8) .. controls (179.16,272.8) and (174.34,264.72) .. (174.34,254.74) -- cycle ;
%Shape: Ellipse [id:dp8353567972638577] 
\draw  [color={rgb, 255:red, 74; green, 144; blue, 226 }  ,draw opacity=1 ] (84,194.43) .. controls (84,179.26) and (92.61,166.96) .. (103.23,166.96) .. controls (113.85,166.96) and (122.46,179.26) .. (122.46,194.43) .. controls (122.46,209.6) and (113.85,221.9) .. (103.23,221.9) .. controls (92.61,221.9) and (84,209.6) .. (84,194.43) -- cycle ;
%Curve Lines [id:da31768787443828983] 
\draw    (138.7,115.23) .. controls (159.87,162.64) and (185.98,166.12) .. (241.03,164.03) ;

%Curve Lines [id:da8202846970423802] 
\draw    (122.11,138.38) .. controls (128.11,148) and (123.88,168.22) .. (102.7,166.82) ;

%Curve Lines [id:da5487750604816559] 
\draw    (102.35,222.04) .. controls (136.58,218.42) and (150.69,206.56) .. (174.69,260.95) ;

%Curve Lines [id:da06243188169244229] 
\draw    (193.75,244.91) .. controls (181.04,210.05) and (212.1,216.32) .. (241.03,219.11) ;

%Shape: Ellipse [id:dp6488804704195652] 
\draw  [color={rgb, 255:red, 208; green, 2; blue, 27 }  ,draw opacity=1 ] (200.29,-7.94) .. controls (200.29,-17.92) and (205.11,-26) .. (211.05,-26) .. controls (217,-26) and (221.81,-17.92) .. (221.81,-7.94) .. controls (221.81,2.03) and (217,10.12) .. (211.05,10.12) .. controls (205.11,10.12) and (200.29,2.03) .. (200.29,-7.94) -- cycle ;
%Shape: Ellipse [id:dp40814442136446105] 
\draw  [color={rgb, 255:red, 74; green, 144; blue, 226 }  ,draw opacity=1 ] (165,61.43) .. controls (165,46.26) and (173.61,33.96) .. (184.23,33.96) .. controls (194.85,33.96) and (203.46,46.26) .. (203.46,61.43) .. controls (203.46,76.6) and (194.85,88.9) .. (184.23,88.9) .. controls (173.61,88.9) and (165,76.6) .. (165,61.43) -- cycle ;
%Shape: Ellipse [id:dp10352071091137627] 
\draw  [color={rgb, 255:red, 74; green, 144; blue, 226 }  ,draw opacity=1 ] (304.04,58.64) .. controls (304.04,43.47) and (312.65,31.17) .. (323.27,31.17) .. controls (333.89,31.17) and (342.5,43.47) .. (342.5,58.64) .. controls (342.5,73.81) and (333.89,86.11) .. (323.27,86.11) .. controls (312.65,86.11) and (304.04,73.81) .. (304.04,58.64) -- cycle ;
%Curve Lines [id:da2619768028913516] 
\draw    (219.7,-17.77) .. controls (240.87,29.64) and (266.98,33.12) .. (322.03,31.03) ;

%Curve Lines [id:da49729773011058176] 
\draw    (203.11,5.38) .. controls (209.11,15) and (204.88,35.22) .. (183.7,33.82) ;

%Curve Lines [id:da5446218255758594] 
\draw    (183.35,89.04) .. controls (217.58,85.42) and (231.69,73.56) .. (255.69,127.95) ;

%Curve Lines [id:da2663893645898455] 
\draw    (274.75,111.91) .. controls (262.04,77.05) and (293.1,83.32) .. (322.03,86.11) ;

%Shape: Ellipse [id:dp603474257443366] 
\draw  [color={rgb, 255:red, 208; green, 2; blue, 27 }  ,draw opacity=1 ] (364.34,377.74) .. controls (364.34,367.77) and (369.16,359.68) .. (375.1,359.68) .. controls (381.05,359.68) and (385.86,367.77) .. (385.86,377.74) .. controls (385.86,387.72) and (381.05,395.8) .. (375.1,395.8) .. controls (369.16,395.8) and (364.34,387.72) .. (364.34,377.74) -- cycle ;
%Shape: Ellipse [id:dp2915042553588931] 
\draw  [color={rgb, 255:red, 74; green, 144; blue, 226 }  ,draw opacity=1 ] (274,317.43) .. controls (274,302.26) and (282.61,289.96) .. (293.23,289.96) .. controls (303.85,289.96) and (312.46,302.26) .. (312.46,317.43) .. controls (312.46,332.6) and (303.85,344.9) .. (293.23,344.9) .. controls (282.61,344.9) and (274,332.6) .. (274,317.43) -- cycle ;
%Shape: Ellipse [id:dp6768032593024496] 
\draw  [color={rgb, 255:red, 74; green, 144; blue, 226 }  ,draw opacity=1 ] (413.04,314.64) .. controls (413.04,299.47) and (421.65,287.17) .. (432.27,287.17) .. controls (442.89,287.17) and (451.5,299.47) .. (451.5,314.64) .. controls (451.5,329.81) and (442.89,342.11) .. (432.27,342.11) .. controls (421.65,342.11) and (413.04,329.81) .. (413.04,314.64) -- cycle ;
%Curve Lines [id:da30365520890128483] 
\draw    (328.7,238.23) .. controls (349.87,285.64) and (375.98,289.12) .. (431.03,287.03) ;

%Curve Lines [id:da8640725552149322] 
\draw    (312.11,261.38) .. controls (318.11,271) and (313.88,291.22) .. (292.7,289.82) ;

%Curve Lines [id:da022262374787747063] 
\draw    (292.35,345.04) .. controls (326.58,341.42) and (340.69,329.56) .. (364.69,383.95) ;

%Curve Lines [id:da6114249631327392] 
\draw    (383.75,367.91) .. controls (371.04,333.05) and (402.1,339.32) .. (431.03,342.11) ;

\end{tikzpicture}

    \caption{$Y_1$ for genus $g=2$}
    \label{fig:X1}
\end{figure}

Applying the isoperimetric inequality to $Y_n$ we have

\begin{equation}
    x_n|X_0| = |Y_n| \leq \frac12|\partial Y_n| = \frac12 (x_n.4\pi(g-1) + d_n.\sum_{i=1}^g 2|D_i|) 
\end{equation}
where we are seeing $Y_n$ as union of copies of $X_0$. Hence we have

\begin{equation}
    |X_0| \leq 2\pi(g-1) + \frac{d_n}{x_n}\sum_{i=1}^g |D_i|,
\end{equation}
so by sending $n$ to infinity we have

\begin{equation}\label{eq:1stbound}
    |X_0| \leq 2\pi(g-1) + \frac{g-1}{g} \sum_{i=1}^g |D_i|
\end{equation}
%We can do a similar construction by only adding copies of $X_0$ adjacent through $D^\pm_g$ (see Figure \ref{fig:betterX1}). After the previous analysis, it is easy to see that the bound for $X_0$ will be 
Finally, we could also look at the bound obtained by adding copies of $X_0$ adjacent through $D^\pm_i$ for $g-k+1\leq i \leq g$. Here we will obtain

\begin{equation}\label{eq:3rdbound}
    |X_0| \leq 2\pi(g-1) + \sum_{i=1}^{g-k}|D_i| + \frac{k-1}{k}\sum_{i=g-k+1}^g |D_i|
\end{equation}
As long as we order our labeling so that the area $|D_i|$ is increasing on $i$ (or increasing on any bound we have available on those areas) Inequality (\ref{eq:2ndbound}) is better than either (\ref{eq:1stbound}), (\ref{eq:3rdbound}). For this reason we will use equation (\ref{eq:2ndbound}) as our improved isoperimetric inequality.
\end{remark}

\begin{remark}
The maximization process of Step 5 is also motivated from the following description for a surface of genus $2$. In this case $\Gamma$ is just one curve, so $L(\Sigma)$ has a similar (but crucially different) definition to the systolic ratio. The systolic ratio is defined as the supremum the quotient $\frac{{\rm systole}^2}{{\rm Area}}$ for all metrics in the surface. In the case $g=2$ the systolic ratio is bounded above by the relatively small constant $ \frac{2}{\sqrt3}$. Unfortunately the systolic ratio and the extremal length systole (the supremum over moduli space of the smallest extremal length) satisfy $SR<ELS$, which is not helpful to bound $L(\Sigma)$ from above. Regardless, if the values of $L(\Sigma)$ are relatively small at least when genus is small, one would answer Maldacena's question for a large portion of the moduli space. 
\end{remark}

Since the square root of the systolic ratio $\frac{{\rm systole}}{\sqrt{{\rm Area}}}$ grows as $\log(g)/\sqrt{g}$ (see [\cite{Gromov83}, Section 5.3] for an upper bound and [\cite{BuSa94}, Equation 1.13] for a lower bound), then $\sup L(\Sigma)$ should grow at least as $\sqrt{g}\log (g)$. Given the role of $L(\Sigma)$ on the inequalities of Theorem \ref{thm:main}, we are interested to determine the order of growth of the bound in (\ref{eq:thm1}). Hence the following questions are very natural.

\begin{que}
What is the maximum value of $L(\Sigma)$ for a given genus? How does this quantity behave asymptotically as $g\rightarrow\infty$?
\end{que}

Observe that while the bound $\sqrt{\frac{21}{2}}(g-1)^{3/2}e^{\frac{21(g-1)}{4}}$ is explicit on the genus, it is not likely to be optimal. We expect that that this is the case since  \cite[Corollary 3]{Maskit} applies to all simple closed curves, while \emph{shortest} curves have often better estimates for extremal length (see for instance \cite[Equation 3.12]{BuSa94} for the hyperbolic systole).

Finally, since $g^\frac12 < g^\frac12 \log (g)$, one would need a stronger approach to solve Question \ref{que:VRneg} for large genus. On the other hand, one can wonder if pulling tight certain inequalities in our approach would yield an answer for small genus.

\begin{que}
Since the isoperimetric and length bounds we are using are not on configurations that realize equality, can we pull tight the inequalities to fully answer Maldacena's question, at least for $g=2$?
\end{que}

\bibliographystyle{amsalpha}
\bibliography{mybib}

\end{document}